\documentclass[12pt, reqno]{amsart}
\usepackage{amsmath, amsthm, amscd, amsfonts, amssymb, graphicx, color}
\usepackage[bookmarksnumbered, colorlinks, plainpages]{hyperref}
\hypersetup{colorlinks=true,linkcolor=red, anchorcolor=green, citecolor=cyan, urlcolor=red, filecolor=magenta, pdftoolbar=true}
\input{mathrsfs.sty}

\newtheorem{theorem}{Theorem}[section]
\newtheorem{lemma}[theorem]{Lemma}
\newtheorem{proposition}[theorem]{Proposition}
\newtheorem{cor}[theorem]{Corollary}
\theoremstyle{definition}

\theoremstyle{remark}
\newtheorem{remark}[theorem]{\bf{Remark}}
\numberwithin{equation}{section}
\begin{document}

\title{Euclidean operator radius and numerical radius inequalities}

\author[S. Jana, P. Bhunia, K. Paul] {Suvendu Jana, Pintu Bhunia, Kallol Paul}

\address[Jana]{Department of Mathematics, Mahisadal Girls' College, Purba Medinipur 721628, West Bengal, India}
\email{janasuva8@gmail.com} 

\address[Bhunia]{Department of Mathematics, Indian Institute of Science, Bengaluru 560012, Karnataka, India}
\email{pintubhunia5206@gmail.com}

\address[Paul]{Department of Mathematics, Jadavpur University, Kolkata 700032, West Bengal, India}
\email{kalloldada@gmail.com}

\thanks{Dr. Pintu Bhunia would like to thank SERB, Govt. of India for the financial support in the form of National Post Doctoral Fellowship (N-PDF, File No. PDF/2022/000325) under the mentorship of Prof. Apoorva Khare}

%\thanks{ }

\renewcommand{\subjclassname}{\textup{2020} Mathematics Subject Classification}

\maketitle

\begin{abstract}
	Let $T$ be a bounded linear operator on a complex Hilbert space $\mathscr{H}.$ 
We obtain various lower and upper bounds for the numerical radius of $T$  by developing  the Euclidean operator radius bounds of a pair of operators, which are stronger than the existing ones. In particular, we develop an inequality that improves on the inequality
$$	w(T) \geq \frac12 {\|T\|}+\frac14  {\left|\|Re(T)\|-\frac12 \|T\| \right|} + \frac14 { \left| \|Im(T)\|-\frac12 \|T\| \right|}.$$
Various equality conditions of the existing numerical radius inequalities are also provided. Further, we study the numerical radius inequalities of $2\times 2$ off-diagonal operator matrices.
Applying the numerical radius bounds of operator matrices, we develop the upper bounds of $w(T)$ by using $t$-Aluthge transform. In particular, we improve the well known inequality 
 $$	w(T) \leq   \frac12 {\|T\|}+ \frac12{ w(\widetilde{T})}, $$
where  $\widetilde{T}=|T|^{1/2}U|T|^{1/2}$ is the Aluthge transform of $T$ and $T=U|T|$ is the polar decomposition of $T$.

%By considering $\alpha=\max \left\{\|Re(T)\|, \frac{\|T\|}{2}  \right\}$ and $\beta=\max \left\{ \|Im(T)\|, \frac{\|T\|}{2}  \right\},$
%we show that the numerical radius  $w(T)$ satisfies 
%\begin{eqnarray*}
%	w(T) &\geq& 	\frac{\|T\|}2+\frac{ \left|\|Re(T)\|-\frac12 \|T\| \right|}4 + \frac{ \left| \|Im(T)\|-\frac12 \|T\| \right|}4 + \frac{|\alpha-\beta|}{2},
%\end{eqnarray*}
%which is stronger than the bound
%$w(T) \geq	\frac{\|T\|}2+\frac{ \left|\|Re(T)\|-\frac12 \|T\| \right|}4 + \frac{ \left| \|Im(T)\|-\frac12 \|T\| \right|}4,$
%given in [Hirzallah et al., Integr. Equ. Oper. Theory, 71 (2011), 129--147].
 %and $|T|=(T^*T)^{1/2}$.
\end{abstract}

\section{Introduction}
\noindent Let $ \mathbb{B}(\mathscr{H})$ denote the $C^*$-algebra of all bounded linear operators on a complex Hilbert $(\mathscr{H},\langle . , . \rangle)$.  For $T\in \mathbb{B}(\mathscr{H}),$ $T^*$ denotes the adjoint of $T$ and  $|T|=({T^*T})^{\frac{1}{2}}$. The Cartesian decomposition of $T$ is $T= Re (T) + i Im(T),$ where $Re(T)=\frac{1}{2}(T+T^*)$ and $Im(T)=\frac{1}{2 i}(T-T^*)$. 
For $0\leq t\leq 1,$ the $t$-Aluthge transform of $T\in \mathbb{B}(\mathscr{H})$ is defined as $\widetilde{T_t}=|T|^{t}U|T|^{1-t},$
where $T=U|T|$ is the polar decomposition of $T$ and $U$ is the partial isometry. In particular, for $t=\frac12$, let $\widetilde{T}=\widetilde{T_{\frac12}}=|T|^{1/2}U|T|^{1/2}$ be the Aluthge transform of $T.$
The numerical radius of $T$, denoted by $w(T)$, is defined as $w(T)=\sup \left \{|\langle Tx,x \rangle|: x\in \mathscr{H}, \|x\|=1 \right \}.$
 The numerical radius $ w(\cdot)$, defines a norm on $\mathbb{B}(\mathscr{H})$, which satisfies 
\begin{eqnarray}\label{eqv}
\frac{1}{2} \|T\|\leq w({T})\leq\|T\|.
\end{eqnarray}
For further information on the numerical radius and related inequalities improving  \eqref{eqv}, we refer to \cite{AK1,MIA20,Book1,PSMK, RIM21,BSM21,LAA2019,STU03,Moradi2020,Book2}.  
Based on the importance of the concept of numerical radius, various generalizations have been studied for the last few years. Such a generalization is the Euclidean operator radius, see \cite{P}.
For $B,C\in \mathbb{B}(\mathscr{H})$, the Euclidean operator radius of  $B$ and $C$, denoted by $w_e(B,C),$ is defined as $$w_e(B,C)= \sup \left \{ \sqrt{|\langle B x,x\rangle|^2+|\langle C x,x\rangle|^2} : x\in \mathscr{H}, \|x\|=1 \right \}.$$
The Euclidean operator radius $w_e(.,.)$, defines a norm on $\mathbb{B}^2(\mathscr{H}) (=\mathbb{B}(\mathscr{H}) \times \mathbb{B}(\mathscr{H}) ),$  which satisfies 
the inequality (see \cite{P})
 \begin{eqnarray}
\frac{{1}}{8}\|B^*B+C^*C\|\leq w_e^2(B,C)\leq\|B^*B+C^*C\|.
\label{eqn1}\end{eqnarray}
 Here  
the constants $\frac{{1}}{8}$ and $1$ are best possible.
 In \cite[Th. 1]{D}, Dragomir proved that 
\begin{eqnarray}\label{D06}
\frac12 w\left(B^2+C^2\right)	\leq w^2_e(B,C)
\end{eqnarray}
 and the constant $\frac12$ is best possible. See \cite{SPK,SMS, SAH21} for more  generalizations on the Euclidean operator radius and related results.

 \smallskip
 In \cite{SPK}, authors studied improvements of the inequalities  \eqref{eqn1} and \eqref{D06}. In this  article we continue the study in that direction.
We obtain lower and upper bounds for the Euclidean operator radius  of a pair of bounded linear operators $B$ and $C$, which improve on the earlier related bounds. From the Euclidean operator radius bounds we develop various lower and upper bounds for the numerical radius of a bounded linear operator $T$, which improve \eqref{eqv} and  the inequality $\frac14\|T^*T+TT^*\|\leq w^2(T)\leq \frac12\|T^*T+TT^*\|$, given in \cite{E}. We study equality conditions of the existing numerical radius inequalities of a bounded linear operator $T$. Further, we obtain numerical radius bounds for the $2\times 2$ off-diagonal operator matrices, which generalize and improve on the existing ones. Applying the numerical radius bounds of $2\times 2$ off-diagonal operator matrices, we obtain an upper bound for the numerical radius of a bounded linear operator $T$ by using $t$-Aluthge transform, which improves and generalizes the bound  $w(T) \leq  \frac12 \|T\|+ \frac12 w(\widetilde{T}),$ given in \cite{STU07}.

\section{Main Results}
We begin with  the following proposition that gives lower bounds for the Euclidean operator radius $w_e(B,C).$

\begin{proposition}\label{lemm1}
	Let  $B,C \in\mathbb{B}(\mathscr{H})$. Then the following inequalities hold:\\
	(i) $w_e(B,C) \geq  \max\{ w(B),w(C)\}.$\\
	(ii) $w_e(B,C) \geq  \frac{1}{\sqrt{2}}w \left(B+e^{i\theta}C\right)\,\,\textit{for all $\theta \in \mathbb{R}$}.$\\
	(iii) $w_e(B,C) \geq  \sqrt{ \frac12 w\left(B^2+e^{i\theta}C^2\right)+\frac12 \left| w^2(B)-w^2(C)\right| } $ \,\,\textit{for all $\theta \in \mathbb{R}$}.\\
	(iv) $w_e(B,C) \geq  \sqrt{\frac12 w\left(BC+CB\right)}.$
\end{proposition}
	
	\begin{proof}
		(i) Follows trivially from the definition.\\
		(ii) We have
		\begin{eqnarray*}
			w_e(B,C) &=&\sup_{\|x\|=1} \sqrt{|\langle Bx,x\rangle|^2+|\langle Cx,x\rangle|^2}\\
			&\geq& \sup_{\|x\|=1} \sqrt{\frac12 \left( |\langle Bx,x\rangle|+|\langle Cx,x\rangle| \right)^2}\\
			&\geq& \sup_{\|x\|=1} \sqrt{\frac12 \left( |\langle Bx,x\rangle+ e^{i\theta}\langle Cx,x\rangle| \right)^2}\\
			&=&\frac{1}{\sqrt{2}}w \left(B+e^{i\theta}C\right).
			\end{eqnarray*}
	(iii) From (i), we have
	\begin{eqnarray*}
		w_e^2(B,C) &\geq& \max \left\{ w^2(B),w^2(C)\right\}\\
		&=& \frac12 (w^2(B)+w^2(C))+ \frac12 |w^2(B)-w^2(C)|\\
		&\geq& \frac12 \left(w(B^2)+w(C^2) \right)+ \frac12 |w^2(B)-w^2(C)|\\
		&\geq& \frac12 w(B^2+e^{i\theta}C^2) + \frac12 |w^2(B)-w^2(C)|.
	\end{eqnarray*}
		(iv) From (ii), we have $w_e(B,C)\geq \frac{1}{\sqrt{2}}w \left(B+C\right)$ and	$w_e(B,C)\geq \frac{1}{\sqrt{2}}w \left(B-C\right).$ Thus,
	\begin{eqnarray*}
		2w_e^2(B,C) &\geq& \frac{1}{{2}}w ^2 \left(B+C\right)+\frac{1}{{2}}w ^2 \left(B-C\right)\\
		&\geq & \frac{1}{{2}}w \left( (B+C )^2\right)+\frac{1}{{2}}w  \left( (B-C)^2\right)\\
		&\geq & \frac{1}{{2}}w \left( (B+C )^2- (B-C)^2\right)\\
		&=& w\left(BC+CB \right).
	\end{eqnarray*}

This completes the proof.	
	\end{proof}

Clearly, Proposition \ref{lemm1} (iii) generalizes and improves the inequality  $w_e(B,C)\geq \sqrt{\frac12 w\left(B^2+C^2\right)}$, proved in \cite[Th. 1]{D}. 
Now, by using Proposition \ref{lemm1} we prove the following theorem.

\begin{theorem}\label{th1}
	 Let  $B,C \in\mathbb{B}(\mathscr{H})$. Then 
\begin{eqnarray*}
\sqrt{\frac14 w(B^2+C^2)+ \frac14 \left(w^2(B)+w^2(C) \right)+ \frac12 \left| w^2(B)-w^2(C)\right|}	&\leq & w_{e}(B,C).
\end{eqnarray*}
\end{theorem}

\begin{proof}
	Take $t_{1}=\max \left\{ w^2(B),\frac12 w\left(B^2+C^2\right) \right\},$ $t_{2}=\max \left\{ w^2(C),\frac12 w\left(B^2+C^2\right) \right\},$ $m_1=\left| w^2(B)-\frac12 w\left(B^2+C^2\right)\right|$ and $m_2=\left| w^2(C)-\frac12 w\left(B^2+C^2\right)\right|.$ 
From the inequalities (i) and (iii) of Proposition \ref{lemm1}, we have
\begin{eqnarray*}
w_{e}^2(B,C)&\geq&\max \{t_1,t_2\}\\
&=&\frac12(t_1+t_2)+\frac12\big|t_1-t_2\big|\\
&=& \frac14 \left( w^2(B)+w^2(C)\right)+\frac14 w(B^2+C^2)+\frac14 (m_1+m_2)+\frac12 \big|t_1-t_2\big|\\
&\geq& \frac14\left( w(B^2)+w(C^2)\right)+\frac14 w(B^2+C^2)+\frac14 (m_1+m_2)+\frac12 \big|t_1-t_2\big|\\
&\geq& \frac14 w(B^2+C^2)+\frac14 w(B^2+C^2)+\frac14 (m_1+m_2)+\frac12 \big|t_1-t_2\big|\\
&=&\frac12 w(B^2+C^2)+\frac14 (m_1+m_2)+\frac12 \big|t_1-t_2\big|\\
&=&\frac14 w(B^2+C^2)+ \frac14 \left(w^2(B)+w^2(C) \right)+ \frac12 \left| w^2(B)-w^2(C)\right|,
\end{eqnarray*}
as desired.
\end{proof}

\begin{remark}
	(i) Clearly, the inequality obtained in Theorem \ref{th1}  is a refinement of  the inequality $ \sqrt{\frac12 w\left(B^2+C^2\right)} \leq w_e(B,C)$, given in \cite[Th. 1]{D}.\\
	(ii) If  $w_e(B,C)=\sqrt{\frac12 w\left(B^2+C^2\right)}$, then  from Theorem \ref{th1} it follows that $w(B)=w(C)=\sqrt{\frac12 w\left(B^2+C^2\right)}$. However, the converse, in general, may not hold. As for example, considering  a non-zero normal operator $B=C,$ we get $w(B)=w(C)=\sqrt{\frac12 w\left(B^2+C^2\right)}$, but $w_e(B,C)= \sqrt2 w(B) \neq w(B)=\sqrt{\frac12 w\left(B^2+C^2\right)}.$\\
	(iii) If 	$w_e(B,C)=\sqrt{\frac12 w\left(B^2+C^2\right)+ \frac12\left|w^2(B)-w^2(C)\right|}$
	then  from Theorem \ref{th1} it follows that $w\left(B^2+C^2\right)=w^2(B)+w^2(C)$ and $w_e(B,C)=\max\{w(B), w(C)  \}.$ The converse of the result is also valid.
\end{remark}

 As an immediate consequence of  Theorem \ref{th1} we have the following result.

\begin{cor}\label{eq2}
	Let $ B,C\in\mathbb{B}(\mathscr{H})$ be normal, then
\begin{eqnarray*}
 w_{e}(B,C) &\geq &
	\sqrt{\frac14 \|B^2+C^2\|+ \frac14 \left(\|B\|^2+ \|C\|^2 \right)+ \frac12 \left| \|B\|^2-\|C\|^2\right|}\\
&=&\sqrt{\frac{1}{2} \left\|B^2+C^2 \right\|+\frac14(p_1+p_2)+\frac12\big|s_1-s_2\big|},
\end{eqnarray*} 
where $s_{1}=\max \left\{ \|B\|^2,\frac12 \|B^2+C^2\| \right\},$ $s_{2}=\max \left\{ \|C\|^2,\frac12 \|B^2+C^2\| \right\},$ $p_1=\big| \|B\|^2-\frac12 \|B^2+C^2\|\big|$ and $p_2=\big| \|C\|^2-\frac12 \|B^2+C^2\|\big|.$ 
\end{cor}

Corollary \ref{eq2} is better than the first inequality in (\ref{eqn1}) when $B$ and $C$ are self-adjoint operators.
From Corollary \ref{eq2} we obtain the following numerical radius bound of a bounded linear operator $T$.  

\begin{cor} \label{pcor}
	Let $T\in \mathbb{B}(\mathscr{H})$. Then 
	\begin{eqnarray*} 
		\sqrt{\frac{1}{8} \|T^*T+TT^*\|+\frac14\left( \|Re(T)\|^2+\|Im(T)\|^2 \right)+\frac12\left|\|Re(T)\|^2-\|Im(T)\|^2\right|}\leq w(T),
	\end{eqnarray*}

\end{cor}
\begin{proof}
	  Considering $ B={Re}(T)$ and $ C=Im(T)$  in Corollary \ref{eq2} we obtain that 
	\begin{eqnarray*} %\label{T}
		w(T)&\geq& \sqrt{\frac{1}{8} \|T^*T+TT^*\|+\frac14\left( \|Re(T)\|^2+\|Im(T)\|^2 \right)+\frac12\left|\|Re(T)\|^2-\|Im(T)\|^2\right|}\\
		&=& \sqrt{\frac{1}{4}  \|T^*T+TT^*\|+\frac14(\alpha+\beta)+\frac12\big|\gamma-\delta\big|},
		\end{eqnarray*} 
	where $\alpha=\big|\|Re(T)\|^2-\frac14\|T^*T+TT^*\|\big|,$ $\beta=\big|\|Im(T)\|^2-\frac14\|T^*T+TT^*\|\big|,$ $\gamma=\max \left\{\|Re(T)\|^2,\frac14 \|TT^*+T^*T\| \right\}$ and $\delta=\max \left\{ \|Im(T)\|^2,\frac14 \|TT^*+T^*T\| \right\}.$
\end{proof}

\begin{remark}
(i) Clearly, the bound obtained in Corollary \ref{pcor} is stronger than the bound obtained in \cite[Th. 2.9]{LAA21}, which is, 
 \begin{eqnarray} \label{T}
	\sqrt{\frac14\|T^*T+TT^*\|+\frac12 \big| \|Re(T)\|^2-\|Im(T)\|^2\big|} \leq w(T).
\end{eqnarray}

  (ii) From Corollary \ref{pcor} it follows that, if $\sqrt{\frac14\|T^*T+TT^*\|+\frac12 \big| \|Re(T)\|^2-\|Im(T)\|^2\big|} = w(T),$ then  $\frac12 \|T^*T+TT^*\|= \|Re(T)\|^2+\|Im(T)\|^2$ and $w(T)=\max\{\|Re(T)\|, \|Im(T)\|\}$. The converse also holds.
\end{remark}

Using Corollary \ref{eq2} we also obtain the following lower bound for the numerical radius.
\begin{cor}\label{pcor1}
	Let $T\in \mathbb{B}(\mathscr{H})$, then
	\begin{eqnarray*}
		\sqrt{\frac18\|T^*T+TT^*\|+\frac18\left(\|Re(T)+Im(T)\|^2+ \|Re(T)-Im(T)\|^2\right)+ \frac{\beta}4 }& \leq & w(T),
	\end{eqnarray*}
where $\beta=\big|\|Re(T)+Im(T)\|^2- \|Re(T)-Im(T)\|^2\big|.$

\end{cor}
\begin{proof}
Considering $ B=\frac{Re(T)+Im(T)}{\sqrt{2}}$ and $ C=\frac{Re(T)-Im(T)}{\sqrt{2}}$ in Corollary \ref{eq2}, we have
	\begin{eqnarray*}
		w(T)& \geq & \sqrt{\frac18\|T^*T+TT^*\|+\frac18\left(\|Re(T)+Im(T)\|^2+ \|Re(T)-Im(T)\|^2\right)+ \frac{\beta}4 }\\
		&=& \sqrt{\frac{1}{4} \|T^*T+TT^*\|+\frac14(\gamma+\delta)+\frac12\big|\xi-\eta\big|},
	\end{eqnarray*}
	where 
	\begin{eqnarray*}
		\gamma &= &\left|\frac{\|Re(T)+Im(T)\|^2}{2}-\frac14\|T^*T+TT^*\|\right|,\\
		\delta &= &\left|\frac{\|Re(T)-Im(T)\|^2}{2}-\frac14\|T^*T+TT^*\|\right|,\\
		\xi  &=  &\max \left\{\frac{\|Re(T)+Im(T)\|^2}{2},\frac14 \|TT^*+T^*T\| \right\},\\
		\eta&=&\max \left\{\frac{\|Re(T)-Im(T)\|^2}{2},\frac14 \|TT^*+T^*T\| \right\}.
	\end{eqnarray*}
	
\end{proof}

\begin{remark}
(i) In \cite[Th. 2.3]{psk1} authors developed the inequality 
\begin{eqnarray*}\label{0pp}
	\sqrt{\frac{1}{4}\|T^*T+TT^*\|+\frac{1}{4}\left| \|Re(T)+Im(T)\|^2-\|Re(T)-Im(T)\|^2\right|} \leq w(T).
\end{eqnarray*}
It is easy to conclude that the inequality obtained in 
 Corollary \ref{pcor1} is  a refinement of the above inequalty.\\
(ii) From the inequality developed in Corollary \ref{pcor1}, it follows that if 
$$\sqrt{\frac{1}{4}\|T^*T+TT^*\|+\frac{1}{4}\left| \|Re(T)+Im(T)\|^2-\|Re(T)-Im(T)\|^2\right|} = w(T)$$ then 
$\|T^*T+TT^*\|=\|Re(T)+Im(T)\|^2+ \|Re(T)-Im(T)\|^2$ and 
$$ w(T)=\max \left\{ \frac{\|Re(T)+Im(T)\|}{\sqrt{2}}, \frac{ \|Re(T)-Im(T)\|} {\sqrt{2}}   \right\} .$$
The converse also holds.
\end{remark}

Next, we obtain an upper bound for the Euclidean operator radius $w_e(B,C).$

\begin{theorem}\label{th3}
	If $ B,C\in\mathbb{B}(\mathscr{H})$  then  for all $t \in [0,1], $
	\begin{eqnarray*}
		&& w_{e}(B,C) \\
		&\leq& \left\| t^2 B^*B+(1-t)^2C^*C\right\|^\frac12+  \frac1{\sqrt{2}}\left\lbrace w^2((1-t)B+tC)+ w^2((1-t)B-tC)\right\rbrace^\frac12.
	\end{eqnarray*} 
	In particular, for $t=\frac12$
	\begin{eqnarray}\label{eq22}
		w_{e}(B,C)&\leq& \frac12 \|B^*B+C^*C\|^\frac12+\frac{1}{2\sqrt{2}} \left\lbrace w^2(B+C)+w^2(B-C)\right\rbrace^\frac12. 
	\end{eqnarray}
	\end{theorem}
\begin{proof}
	Take $x\in \mathscr{H}$ with $\|x\|=1$. We have
	\begin{eqnarray*}
		&&\left(|\langle Bx,x\rangle|^2+|\langle Cx,x\rangle|^2\right)^\frac12\\
		&=& \left(|t\langle Bx,x\rangle+(1-t)\langle Bx,x\rangle|^2+|(1-t)\langle Cx,x\rangle+t\langle Cx,x\rangle|^2\right)^\frac12\\
		&\leq& \left( t^2|\langle Bx,x\rangle|^2+(1-t)^2|\langle Cx,x\rangle|^2\right)^\frac12+\left( (1-t)^2|\langle Bx,x\rangle|^2+t^2|\langle Cx,x\rangle|^2\right)^\frac12\\
		&&\,\,\,\,\,\,\,\,\,\,\,\,\,\,\,\,\,\,\,\,\,\,\,\,\,\,\,\,\,\,\,\,\,\,\,\,\,\,\,\,\,\,\,\,\,\,\,\,\,\,\,\,\,\,\,\,\,\,\,\,\,\,\,\,\,\,\,\,\,\,\,\,\,\,\,\,\,\,\,\,\,\,\,\,\,\,\,\,\,\,\,\,\,\,\,\,\,\,\,\,\,\,\,\,\,\,\,\,\,\,\,(\textit{by Minkowski inequality})\\
		&\leq& \left( t^2\|Bx\|^2+(1-t)^2\|Cx\|^2\right)^\frac12\\
		&&+ \left(\frac12|\langle\left((1-t)B+tC\right)x,x\rangle|^2
		+\frac12|\langle\left((1-t)B-tC\right)x,x\rangle|^2\right)^\frac12\\
		&\leq& \left\| t^2 B^*B+(1-t)^2C^*C\right\|^\frac12+\left\lbrace \frac12 w^2((1-t)B+tC)+\frac12 w^2((1-t)B-tC)\right\rbrace^\frac12.
	\end{eqnarray*}
	Taking supremum over all $x \in \mathscr{H}$ with $\|x\|=1$, we get the first inequality. In particular, considering $t=\frac12$ we get the second inequality.
\end{proof}

Our next result reads as follows.
\begin{theorem}\label{corp1}
	Let $T\in \mathbb{ B}(\mathscr{H}),$ then the following inequalities hold:
\[(i) \, \, \sqrt{\frac14\|T^*T+TT^*\|+ \alpha} \leq w(T) \leq \sqrt{\frac14\|T^*T+TT^*\|+ \beta},\]
where $ \alpha = \frac12 \big| \|Re(T)\|^2-\|Im(T)\|^2\big|,$ $ \beta = \frac12 \left( \|Re(T)\|^2+\|Im(T)\|^2˘\right).$
\[(ii) \, \, \sqrt{\frac14\|T^*T+TT^*\|+ \gamma} \leq w(T) \leq \sqrt{\frac14\|T^*T+TT^*\|+ \delta} ,\]
where 
\begin{eqnarray*}
\gamma &=  &\frac{1}{4}\left| \|Re(T)+Im(T)\|^2-\|Re(T)-Im(T)\|^2\right| , \\
\delta & =  & \frac{1}{4} \left(\|Re(T)+Im(T)\|^2+\|Re(T)-Im(T)\|^2\right).
\end{eqnarray*}

\end{theorem}
\begin{proof}
	(i) First inequality follows from  \cite[Th. 2.9]{LAA21} and the second inequality follows from the inequality \eqref{eq22} by considering $B=T$ and $C=T^*.$\\
	(ii)  First inequality follows from \cite[Th. 2.3]{psk1}  and the second inequality follows from the inequality \eqref{eq22} by considering $B= Re(T)$ and $C=Im(T).$
\end{proof}

To present our next result need the following Hermite-Hadamard inequality, see  \cite[p. 137]{PPT}. For a convex function $f:J\rightarrow \mathbb{R}$ and $a,b\in J$,  we have 
 \begin{eqnarray}
 f\left(\frac{a+b}{2}\right) \leq\int_{0}^{1} f(ta+(1-t)b)dt\leq \frac{f(a)+f(b)}{2}.
\label{eq21} \end{eqnarray}

Also, we need the following lemmas.

\begin{lemma}\cite[(4.24)]{MP}  %\cite{SM}
 Let $A\in\mathbb{B}(\mathscr{H})$ be self-adjoint with spectrum contained in the interval $J$ and let $x\in\mathscr{H}$ with $\|x\|=1$. If $f$ is a convex function on $J$, then $$
f (\langle Ax, x\rangle)\leq\langle f(A)x, x\rangle.$$
\label{M}\end{lemma}

\begin{lemma}\cite{kato}(Generalized Cauchy-Schwarz inequality)  %\cite{a1}
 If $A\in\mathbb{B}(\mathscr{H})$, then 
 $$ |\langle Ax,y\rangle|^2\leq\langle|A|^{2\alpha} x,x\rangle\langle|A^*|^{2(1-\alpha)}y,y\rangle,$$ for all $x,y\in\mathscr{H}$ and for all $ \alpha\in [0,1]$.

\label{lem1}\end{lemma}

Now we are in a position to prove our result.

\begin{theorem}\label{cor2}
Let $ B,C\in\mathbb{B}(\mathscr{H})$. If $f:[0,\infty)\rightarrow [0,\infty)$ is an increasing operator convex function, then 
\begin{eqnarray*}
	f\left(w_{e}^2(B,C)\right)&\leq&\left\|\int_{0}^{1} f \left(t(B^*B+C^*C)+(1-t)(BB^*+CC^*)\right)dt\right\|\\
	&\leq& \frac12\left\| f (B^*B+C^*C)+ f(BB^*+CC^*)\right\|.
\end{eqnarray*}
\end{theorem}

\begin{proof}
Take $x\in \mathscr{H}$ with $\|x\|=1$. We have
\begin{eqnarray*}
&&f(|\langle Bx,x\rangle|^2+|\langle Cx,x\rangle|^2)\\
&&\leq f(\langle |B|x,x\rangle\langle |B^*|x,x\rangle+\langle |C|x,x\rangle\langle |C^*|x,x\rangle) \,\,(\textit{by Lemma \ref{lem1}})\\
&&\leq f\left(\{\langle |B|x,x\rangle^2+\langle |C|x,x\rangle^2\}^\frac12\{\langle |B^*|x,x\rangle^2+\langle |C^*|x,x\rangle^2\}^\frac12\right)\\
&&\leq f\left(\frac12\langle(B^*B+C^*C)x,x\rangle+\frac12\langle(BB^*+CC^*)x,x\rangle\right)\\
&&\leq \int_{0}^{1}f\left( \left\langle\big(t(B^*B+C^*C)+(1-t)(BB^*+CC^*)\big)x,x\right\rangle \right) dt \,\,(\mbox{by \eqref{eq21}}).
\end{eqnarray*}
Now, 
\begin{eqnarray*}
&&f\left( \left\langle\big(t(B^*B+C^*C)+(1-t)(BB^*+CC^*)\big)x,x\right\rangle \right)\\
&&\leq\left\langle f \big(t(B^*B+C^*C)+(1-t)(BB^*+CC^*)\big)x,x \right\rangle \,(\textit{by Lemma \ref{M}})\\
&&\leq  t\langle f (B^*B+C^*C)x,x\rangle+(1-t)\langle f(BB^*+CC^*)x,x\rangle,
\end{eqnarray*}
where the last inequality follows form operator convexity of $f$. Therefore,
\begin{eqnarray*}
&&\int_{0}^{1}f\left( \left\langle\big(t(B^*B+C^*C)+(1-t)(BB^*+CC^*)\big)x,x\right\rangle \right) dt\\
&& \leq\langle \int_{0}^{1} f \big(t(B^*B+C^*C)+(1-t)(BB^*+CC^*)\big)dt \, x,x\rangle\\ 
&& \leq \frac12\big(\langle f (B^*B+C^*C)x,x\rangle+\langle f(BB^*+CC^*)x,x\rangle\big).
\end{eqnarray*}
Taking the supremum over $x\in \mathscr{H}$ with $\|x\|=1$, we get 
\begin{eqnarray*}
f\left(w_{e}^2(B,C)\right)&\leq&\left\|\int_{0}^{1} f \big(t(B^*B+C^*C)+(1-t)(BB^*+CC^*)\big)dt\right\|\\
&\leq& \frac12\left\| f (B^*B+C^*C)+ f(BB^*+CC^*)\right\|.
\end{eqnarray*}
Thus, we complete the proof.
 \end{proof}

Since for $1\leq r\leq 2$ the function  $f(x)=x^r$, $x\geq 0$  is an increasing operator convex function, we have
\begin{eqnarray}\label{0000p}
	w_{e}^{2r}(B,C)&\leq&\left\| \int_{0}^{1}\left( t(B^*B+C^*C)+(1-t)(BB^*+CC^*)\right)^r dt \right \|\\
	&\leq&\frac12\|(B^*B+C^*C)^r+(BB^*+CC^*)^r\|.
\end{eqnarray}
 In particular, for $r=1$,
 \begin{eqnarray}\label{00pp}
 	w_{e}^{2}(B,C)&\leq&\left\| \int_{0}^{1}\left( t(B^*B+C^*C)+(1-t)(BB^*+CC^*)\right) dt \right \|\nonumber\\
 	&\leq&\frac12\|(B^*B+C^*C)+(BB^*+CC^*)\|.
 \end{eqnarray}
The above inequality can also be derived from 
 \begin{eqnarray*}\label{eq23} w_{e}^2(B,C)\leq \left\|\alpha (|B|^2+|C|^2)+(1-\alpha)(|B^*|^2+|C^*|^2)\right\|, \, 0\leq \alpha\leq 1,
   \end{eqnarray*}
proved by Moslehian et al. \cite[Prop. 3.9]{SMS}.
Now, if we take $B=C=T$ in \eqref{0000p} we obtain the following numerical radius inequality.

\begin{cor}\label{corppp}
	Let $T\in \mathbb{ B}(\mathscr{H})$, then
	\begin{eqnarray*}
		w^2(T) &\leq& \left\| \int_{0}^{1}\big( tT^*T+(1-t)TT^*\big)^r dt \right \|^{{1}/{r}}
		\,\, \leq \,\, \left\|\frac{(T^*T)^r+(TT^*)^r}{2}\right\|^{{1}/{r}},
	\end{eqnarray*}
for $1\leq r \leq 2.$
\end{cor}

Next, in the following theorem we develop a lower bound for the numerical radius of a bounded linear operator $T$.

\begin{theorem}\label{th11}
	Let $T\in\mathbb{B}(\mathscr{H}),$ then 
	$$\frac{1}{4}\|T\|+\frac14 \left(\|Re(T)\|+\|Im(T)\| \right)+\frac12 \left|\|Re(T)\|-\|Im(T)\| \right| \leq w(T) .$$
\end{theorem}

\begin{proof}
	From the Cartesian decomposition of $T$, it is easy to verify that  $\|Re(T)\|\leq w(T),$ $\|Im(T)\|\leq w(T)$ and $\frac12 \|T\|\leq w(T).$ Take $r_1=  \big|\|Re(T)\|-\frac12 \|T\|\big|$, $r_2=\big| \|Im(T)\|-\frac12 \|T\|\big|,$  $q_1=\max \left\{\|Re(T)\|,\frac12 \|T\| \right\}$ and $q_2=\max \left\{\|Im(T)\|,\frac12 \|T\| \right\}.$
	We have
	\begin{eqnarray*}
		w(T)&\geq&\max\{q_1,q_2\}\\
		&=&\frac12(q_1+q_2)+\frac12 \big|q_1-q_2\big|\\
		&=&\frac14 \|T\|+\frac14(\|Re(T)\|+\|Im(T)\|)+\frac14 (r_1+r_2)+ \frac12 \big|q_1-q_2\big|\\
		&\geq& \frac14 \|T\|+\frac14 \|Re(T)+i Im(T)\|+\frac14 (r_1+r_2)+\frac12 \big|q_1-q_2\big|\\
		&=&\frac12 \|T\|+\frac14 (r_1+r_2)+\frac12 \big|q_1-q_2\big|\\
		&=&\frac12 \|T\|+\frac14 \left|\|Re(T)\|-\frac12 \|T\| \right| +\frac14 \left| \|Im(T)\|-\frac12 \|T\| \right| +\frac12 \big|q_1-q_2\big|\\
		&=&\frac{1}{4}\|T\|+\frac14 \left(\|Re(T)\|+\|Im(T)\| \right)+\frac12 \left|\|Re(T)\|-\|Im(T)\| \right|,
	\end{eqnarray*}
	as desired.
\end{proof} 

\begin{remark}
	(i) It follows from \cite{HKS2011} that 
	\begin{eqnarray}\label{00p}
		\frac12 \|T\|+\frac14 \left|\|Re(T)\|-\frac12 \|T\| \right| + \frac14 \left| \|Im(T)\|-\frac12 \|T\| \right| \leq w(T).
	\end{eqnarray}
	Clearly, the inequality in Theorem \ref{th11} refines the inequality \eqref{00p}.\\
	(ii) It follows from  Theorem \ref{th11} that if   $$\frac12 \|T\|+\frac14 \left|\|Re(T)\|-\frac12 \|T\| \right| + \frac14 \left| \|Im(T)\|-\frac12 \|T\| \right| = w(T)$$ then  $\max \left\{\|Re(T)\|,\frac12 \|T\| \right\}=\max \left\{\|Im(T)\|,\frac12 \|T\| \right\}.$ However, the converse may not be true. \\%As for example, ................
	(iii)  For $T\in \mathbb{B}(\mathscr {H}),$  Bhunia and Paul \cite[Th. 2.1]{LAA21} proved that 
	\begin{eqnarray}\label{ppp}
		\frac12\|T\|+\frac12\big|\|\Re(T)\|-\|\Im(T)\|\big|\leq w(T).
	\end{eqnarray}
	Clearly, the inequality in Theorem \ref{th11} refines  \eqref{ppp}.\\
	(iv)  It follows from Theorem \ref{th11}  that if
	$ w(T)=\frac12\|T\|+\frac12\big|\|\Re(T)\|-\|\Im(T)\|\big|$, then
	 $\|T\|=\|Re(T)\|+\|Im(T)\|$ and $w(T)=\max\{ \|Re(T)\|, \|Im(T)\|\}.$ 
	 The converse is also true.
\end{remark}

%As a consequence of Theorem \ref{th11} we obtain the following proposition.
%\begin{proposition}
%	Let $T\in \mathbb{B}(\mathscr{H})$ be such that  $\frac{1}2\|T\|\leq \min \{\|Re(T)\|, \|Im(T)\| \}$. Then the following conditions are equivalent:\\
%	(i) $\frac12\|T\|+\frac12\big|\|\Re(T)\|-\|\Im(T)\|\big|= w(T).$\\
%	(ii) ${\frac{1}{2}\|T\|} = w(T).$
%\end{proposition}

\section {Numerical radius bounds of $2 \times 2$ operator matrices}

Using the numerical radius inequalities obtained in Section 2, here we develop the numerical radius bounds of $2\times 2$ off-diagonal operator matrices. Suppose $\mathscr{H}\oplus \mathscr{H}$ is the direct sum of two copies of $\mathscr{H}$, and 
 $\begin{pmatrix}
	B & X\\
	Y& C
\end{pmatrix}\in \mathbb{B}(\mathscr{H}\oplus \mathscr{H})$ is a $2 \times 2$ operator matrix, defined by $\begin{pmatrix}
B & X\\
Y& C
\end{pmatrix} \begin{pmatrix}
x\\
y
\end{pmatrix}= \begin{pmatrix}
Bx+Xy\\
Yx+Cy
\end{pmatrix},$  $\forall \begin{pmatrix}
x\\
y
\end{pmatrix}\in \mathscr{H}\oplus \mathscr{H}.$
Considering $T=\begin{pmatrix}
	0 & X\\
	Y& 0
\end{pmatrix}\in \mathbb{B}(\mathscr{H}\oplus \mathscr{H})$ in Theorem \ref{th11}, Corollary \ref{pcor}, Corollary \ref{pcor1} and  Theorem \ref{corp1}  respectively, we get the following bounds for the numerical radius of the $2\times 2$ off-diagonal operator matrix $\begin{pmatrix}
0 & X\\
Y& 0
\end{pmatrix}.$

\begin{theorem}\label{0000pp}
	Let $T=\begin{pmatrix}
		0 & X\\
		Y& 0
	\end{pmatrix}\in \mathbb{B}(\mathscr{H}\oplus \mathscr{H})$, then the following inequalities hold:\\
 \begin{eqnarray*}
(i)\,\,	w\left( T\right) &\geq& \max \left\{\frac{\|X\|}4, \frac{\|Y\|}4 \right\}\\
&&+ \frac14 \left(\frac{\|X+Y^*\|}2+\frac{\|X-Y^*\|}2\right)
 +\frac12\left| \frac{\|X+Y^*\|}2-\frac{\|X-Y^*\|}2\right|.\\
(ii) \,\, w^2\left(T\right) &\geq&  \max \left\{\frac{\|X^*X+YY^*\|}8, \frac{\|XX^*+Y^*Y\|}8 \right\}\\
&& +\frac14\left(\frac{\|X+Y^*\|^2}{4}+\frac{\|X-Y^*\|^2}{4}  \right)
  +\frac12 \left|  \frac{\|X+Y^*\|^2}{4}-\frac{\|X-Y^*\|^2}{4} \right|.\\
(iii) \,\,w^2(T) &\geq & \max \left\{\frac{\|X^*X+YY^*\|}8, \frac{\|XX^*+Y^*Y\|}8 \right\}\\
&&+\frac1{8}\left(\frac{\|(1-i)X+(1+i)Y^*\|^2}4+\frac{\|(1+i)X-(1-i)Y^*\|^2}4  \right)\\
&&  +\frac1{4} \left| \frac{\|(1-i)X+(1+i)Y^*\|^2}4-\frac{\|(1+i)X-(1-i)Y^*\|^2}4 \right|.\\
(iv)\,\, w^2(T) &\leq&   \max \left\{\frac{\|X^*X+YY^*\|}4, \frac{\|XX^*+Y^*Y\|}4 \right\}+ \frac12  \left|  \frac{\|X+Y^*\|^2}{4}+\frac{\|X-Y^*\|^2}{4} \right|.\\
(v)\,\, w^2(T) &\leq&  \max \left\{\frac{\|X^*X+YY^*\|}4, \frac{\|XX^*+Y^*Y\|}4 \right\} \\
&& +\frac1{4} \left| \frac{\|(1-i)X+(1+i)Y^*\|^2}4+\frac{\|(1+i)X-(1-i)Y^*\|^2}4 \right|.
\end{eqnarray*}

\end{theorem}

\begin{remark}
	(i) We remark that the bound in Theorem \ref{0000pp} (i) is stronger than the same in \cite[Th. 2.7]{PKO}, namely,
	 \begin{eqnarray*}
			w\left( T\right) &\geq& \max \left\{\frac{\|X\|}2, \frac{\|Y\|}2 \right\}+\frac12\left| \frac{\|X+Y^*\|}2-\frac{\|X-Y^*\|}2\right|.
	\end{eqnarray*}
(ii) It is easy to verify that the bound in Theorem \ref{0000pp} (ii) is stronger than the same in \cite[Th. 2.12]{PKO}, namely,
\begin{eqnarray*}
w^2\left(T\right) &\geq&  \max \left\{\frac{\|X^*X+YY^*\|}4, \frac{\|XX^*+Y^*Y\|}4 \right\}
+\frac12 \left|  \frac{\|X+Y^*\|^2}{4}-\frac{\|X-Y^*\|^2}{4} \right|.
\end{eqnarray*}
%(iii)..........................\\
%(iv)................................
\end{remark}

Now, by applying the operator matrix technique we develop upper bounds for the numerical radius of a bounded linear operator $T$ by using the $t$-Aluthge transform. First we give the following upper bound for the numerical radius $w\begin{pmatrix}
	0 & X \\
	Y & 0
\end{pmatrix}$, where $X,Y\in\mathbb{B}(\mathscr{H}). $

%The inequalities (\ref{eq31}) and (\ref{eq32}) in the second lemma can be found in and the proof follows from \cite{PSK}.
\begin{theorem}\label{lem22} \cite[Th. 2.5 and Cor. 2.6]{PSK}
	Let $X,Y\in\mathbb{B}(\mathscr{H})$ and $T=\begin{pmatrix}
		0 & X\\
		Y& 0
	\end{pmatrix}\in \mathbb{B}(\mathscr{H}\oplus \mathscr{H})$. If $  S=|X|^2+|Y^*|^2$ and  $P=|X^*|^2+|Y|^2$, then
$$w^2\begin{pmatrix}
	0 & X \\
	Y & 0
\end{pmatrix}=w^2(T) \leq  \sqrt{\min \{ \beta, \gamma  \} },$$
where $$\beta=\frac{1}{16} \|S\|^2+\frac14 w^2(YX)+\frac18 w(YXS+SYX),$$
$$ \gamma= \frac{1}{16} \|P\|^2+\frac14 w^2(XY)+\frac18 w(XYP+PXY).$$
\end{theorem}	

For $X,Y\in\mathbb{B}(\mathscr{H})$, we have the following inequalities:
   \begin{eqnarray}\label{0p}
	w(XY)&\leq&w\begin{pmatrix}
		XY & 0 \\
		0 & YX
	\end{pmatrix}
	= w\left(\begin{pmatrix}
		0 & X \\
		Y & 0
	\end{pmatrix}^2\right)
	\leq w^2\begin{pmatrix}
		0 & X \\
		Y & 0
	\end{pmatrix}.
	\end{eqnarray}
Now, by using \eqref{0p} and Theorem \ref{lem22}, we prove the following result.
\begin{cor}\label{cor99}
	Let $T\in\mathbb{B}(\mathscr{H})$. If $P_t=|T|^{2(1-t)}+|T|^{2t},$ $0\leq t\leq 1,$  then
	\begin{eqnarray}\label{thh22}
		w(T) &\leq &\sqrt{\frac{1}{16}\left\|P_t \right\|^2+\frac14 w^2(\widetilde{T_t})+\frac18 w(\widetilde{T_t}P_t+P_t\widetilde{T_t})}\\
		&\leq&  \frac14\left\||T|^{2(1-t)}+|T|^{2t} \right\|+  \frac12 w(\widetilde{T_t}) \nonumber.
	\end{eqnarray}
	In particular, for $t=\frac12$
	\begin{eqnarray}\label{thh22i}
		w(T) &\leq& \sqrt{\frac{1}{4}\left\|T \right\|^2+\frac14 w^2(\widetilde{T})+\frac14 w(\widetilde{T}|T|+|T|\widetilde{T})} \\
		&\leq & \frac12 \|T\|+ \frac12 w(\widetilde{T}) \nonumber.
	\end{eqnarray}
\end{cor}

\begin{proof}
	Taking $X=U|T|^{1-t}$ and $Y=|T|^t$ in Theorem \ref{lem22} (in the expression $\beta$) we obtain the inequality \eqref{thh22}, and the next inequality follows from the inequality (see \cite{FJ})   $w( XY+Y^*X) \leq 2\|Y\|w(X)$ for all $ X,Y\in\mathbb{B}(\mathscr{H})$. The rest of the inequalitie follows by considering $t=\frac12$.
\end{proof}

\begin{remark}
	(i) Let $T\in\mathbb{B}(\mathscr{H})$. Then, clearly the inequality \eqref{thh22} refines the  bound $w(T)\leq \frac14\left\||T|^{2(1-t)}+|T|^{2t} \right\|+  \frac12 w(\widetilde{T_t}),$ obtained by Kittaneh et al. \cite[Cor. 2.2]{FHM}.\\
	(ii) We would like to remark that the inequality \eqref{thh22i} is stronger than the inequality $w(T)\leq \frac12\left\|T \right\|+  \frac12 w(\widetilde{T}) \,\,\left(\leq \frac12 \|T\|+\frac12 \left\|T^2\right\|^{1/2}\right),$ proved by Yamazaki \cite[Th. 2.1]{STU07}. 
	
\end{remark}
Note that the inequality \eqref{thh22i} is already proved  in \cite[Th. 2.6]{psk1} but the  approach  is different and  simple.
Finally, we prove the following result.
\begin{cor}\label{th00}
	Let $T\in\mathbb{B}(\mathscr{H})$. If $Q_t=|T^*|^{2(1-t)}+|T|^{2t}$, $0\leq t\leq 1,$ then 
	\begin{eqnarray}\label{0099}
		w(T) &\leq & \sqrt{ \frac{1}{16}\left \|Q_t \right\|^2+\frac14 w^2(T)+\frac18 w(TQ_t+Q_tT)}\\
		&\leq& \frac14 \||T^*|^{2(1-t)}+|T|^{2t}\|+\frac12 w(T)\nonumber\\
		&\leq& \frac12 \left \||T^*|^{2(1-t)}+|T|^{2t} \right\|\nonumber.
	\end{eqnarray}
%for every $t\in [0,1].$
	\end{cor} 
\begin{proof}
	Taking $X=U|T|^{1-t}$, $Y=|T|^t$ in Theorem \ref{lem22} (in the expression $\gamma$) we obtain the inequality \eqref{0099}.
	The second inequality follows from $w( XY+Y^*X) \leq 2\|Y\|w(X)$ for all $ X,Y \in\mathbb{B}(\mathscr{H})$ (see \cite{FJ}). The last inequality follows trivially.
\end{proof}

\begin{remark}
	Let $T\in\mathbb{B}(\mathscr{H})$. Then,
clearly the bound in \eqref{0099} is sharper than the bound $ w(T)\leq \frac14 \left\||T^*|^{2(1-t)}+|T|^{2t} \right\|+\frac12 w(T),$ proved by Kittaneh et al. \cite{FHM}. 
\end{remark}

%\noindent \textbf{Competing Interests.}\\
%On behalf of all authors, the corresponding author  declares that there is no financial or non-financial interests that are directly or indirectly related to the work submitted for publication.

\bibliographystyle{amsplain}

\end{document}